\documentclass[11pt,twoside]{amsart}

\usepackage[latin1]  {inputenc}%
\usepackage[T1]      {fontenc }%
\usepackage          {amsmath }%
\usepackage          {amsfonts}%
\usepackage          {amssymb }%
\usepackage          {amsthm  }%
\usepackage          {a4wide  }%
\usepackage          {url     }%
\usepackage          {tikz    }%
\usepackage[bookmarks=false,pdfborder={0 0 0.05}]{hyperref}
\usepackage[all]{xy}

\usepackage{enumerate, amsmath, amsfonts, amssymb, amsthm,  wasysym, graphics, graphicx, xcolor, frcursive,comment,bbm}

\usepackage{etex}
\newcommand{\ts}{\textsuperscript}

\definecolor{darkblue}{rgb}{0.0,0,0.7} 
\definecolor{darkred}{rgb}{0.7,0,0} 

\usepackage{hyperref}
\usepackage[all]{xy}
\usepackage[T1]{fontenc}

\def\defn#1{{\sf #1}}

\newcommand{\ZZ}{\mathbb Z}

\newcommand{\spanz}{\mathrm{span}_{\mathbb{Z}}}

\DeclareMathOperator{\Red}{Red}

\DeclareMathOperator{\idop}{id}
\DeclareMathOperator{\stab}{Stab}

\DeclareMathOperator{\Fac}{Fac}

\DeclareMathOperator{\GL}{GL}
\DeclareMathOperator{\TR}{tr}
\DeclareMathOperator{\Tr}{Tr}

\DeclareMathOperator{\Aff}{Aff}

\newtheorem{theorem}{Theorem}[section]

\newtheorem{Proposition}[theorem]{Proposition}
\newtheorem{Lemma}[theorem]{Lemma}

\theoremstyle{definition}
\newtheorem{Definition}[theorem]{Definition}
\newtheorem{remark}[theorem]{Remark}
\newtheorem{example}[theorem]{Example}

\title[On the Hurwitz action in affine Coxeter groups]{On the Hurwitz action in affine Coxeter groups}

\author[P.~Wegener]{Patrick Wegener}
\address{Patrick Wegener, TU Kaiserslautern, Germany}
\email{wegener@mathematik.uni-kl.de}



\begin{document}

\begin{abstract}
We call an element of a Coxeter group a parabolic quasi-Coxeter element if it has a reduced decomposition into a product of reflections that generate a parabolic subgroup. We show that for a parabolic quasi-Coxeter element in an affine Coxeter group the Hurwitz action on its set of reduced decompositions into a product of reflections is transitive.
\end{abstract}

\maketitle

\tableofcontents

\section{Introduction}\label{sec:intro}

Let $G$ be an arbitrary group and $n$ a positive integer. There is an action of the braid group $\mathcal{B}_n$ on $n$ strands on $G^n$, where the standard generator $\sigma_i\in\mathcal{B}_n$ which exchanges the $i$\ts{th} and $(i+1)$\ts{th} strands acts as 
$$\sigma_i\cdot (g_1,\dots, g_n):=(g_1,\dots, g_{i-1}, g_i g_{i+1}g_i^{-1},g_i,g_{i+2},\dots, g_n).$$
Notice that the product of the entries stays unchanged and that all the tuples in a given orbit generate the same subgroup of $G$. This action is called the \defn{Hurwitz action} since it was first studied by Hurwitz in 1891 \cite{Hur91}, for the case $G=\mathfrak{S}_n$. 

Two elements $g,h \in G^n$ are called \defn{Hurwitz equivalent} if there is a braid $\beta \in \mathcal{B}_n$ such that $\beta(g)=h$. It has been shown by Liberman and Teicher \cite{LT} that the question of whether two elements in $G^n$ are Hurwitz equivalent or not is undecidable in general. Nevertheless there are results in some cases like generalized quaternion groups, (semi-)dihedral groups or dicyclic groups, see \cite{Hou08,Sia09}.
Certainly the Hurwitz action also plays a role in algebraic geometry, in particular in the braid monodromy of a projective curve, see for instance \cite{KT00, Bri88}.

A Coxeter group $W$ (resp. a Coxeter system $(W,S)$ with $S$ a set of simple reflections) is a generalization of a reflection group. It is natural to decompose an element $w \in W$ into a product of reflections, where the set of reflections for $(W,S)$ is given by $T=\{wsw^{-1} \mid s \in S,~w \in W\}$. We call such a decomposition a \defn{reflection decomposition} and the pair $(W,T)$ is called a \defn{dual Coxeter system} (see Section \ref{sec:DualCoxeter} for a precise definition).
The Hurwitz action can be restricted to the set of reflection decompositions of a fixed element $w$. Given a reflection decomposition $(t_1,\dots, t_k)$ of $w$, that is $t_1,\dots, t_k \in T$ and $w=t_1\cdots t_k$, the generator $\sigma_i\in\mathcal{B}_n$ then acts as
$$\sigma_i\cdot (t_1,\dots, t_k):=(t_1,\dots, t_{i-1}, t_i t_{i+1}t_i,t_i,t_{i+2},\dots, t_k).$$

The right-hand side is again a reflection decomposition of $w$. 
Further, the property that a reflection decomposition is \defn{reduced} (that is, of minimal length) is preserved by this action.
From now on, we will take Hurwitz action to mean the Hurwitz action on reflection decompositions. 

It is a well known property of the Coxeter system $(W,S)$ that two reduced decompositions of an element $w \in W$ into a product of simple reflections can be transformed into each other by succesive use of the braid relations. This property is due to Matsumoto and therefore also called \defn{Matsumoto property}. Now, considering again a reflection decomposition $(t_1, \ldots, t_k)$ of $w$, we see that the braid group generator $\sigma_i$ acts on the $i$\ts{th} and $(i+1)$\ts{th} entries by replacing $(t_i, t_{i+1})$ by $(t_i t_{i+1}t_i, t_i)$ which corresponds exactly to a \defn{dual braid relation} in the sense of Bessis \cite{Bes03}. Hence determining whether one can pass from any reduced reflection decomposition of an element to any other just by applying a sequence of dual braid relations is equivalent to determining whether the Hurwitz action on the set of reduced reflection decompositions of the element is transitive. Therefore the transitivity of the Hurwitz action should be seen as a kind of Matsumoto property for the generating set $T$.

The transitivity of the Hurwitz action on the set of reduced reflection decompositions has long been known to hold for a family of elements commonly called \defn{parabolic Coxeter elements} (note that there are several unequivalent definitions of these elements in the literature). For more on the topic we refer to \cite{BDSW14}, and the references therein, where a simple proof of the transitivity of the Hurwitz action was shown for (suitably defined) parabolic Coxeter elements in a (not necessarily finite) Coxeter group. For Coxeter elements, the transitivity was first shown by Igusa and Schiffler \cite{IS10}. The Hurwitz action in Coxeter groups has also been studied for nonreduced reflection decompositions \cite{LR16, WY} and outside of the context of parabolic Coxeter elements \cite{Voi85, Mi06}. 

In \cite{BGRW17} the authors provided a necessary and sufficient condition on an element of a finite Coxeter group to ensure the transitivity of the Hurwitz action on its set of reduced reflection decompositions. An element of a Coxeter group is called a \defn{parabolic quasi-Coxeter element} if it admits a reduced reflection decomposition which generates a parabolic subgroup. Given a finite dual Coxeter system $(W,T)$ and an element $w \in W$, it is shown in \cite[Theorem 1.1]{BGRW17} that the Hurwitz action is transitive on the set of reduced reflection decompositions of $w$, denoted $\Red_T(w)$, if and only if $w$ is a parabolic quasi-Coxeter element for $(W,T)$.

In this paper we show that the sufficient condition also applies to affine Coxeter groups. The main result is the following.

\begin{theorem}
\label{th:themaintheorem1}
  Let $(W,T)$ be a dual Coxeter system with $W$ an affine Coxeter group and 
  let~$w \in W$.
  If $w$ is a parabolic quasi-Coxeter element for $(W,T)$, then the Hurwitz action on $\Red_{T}(w)$ is transitive.
\end{theorem}

Note that the proper parabolic subgroups of an affine Coxeter group are all finite. Therefore, if $w$ is a parabolic quasi-Coxeter element for a proper parabolic subgroup, the result of Theorem \ref{th:themaintheorem1} is already covered by the corresponding result in finite Coxeter groups \cite[Theorem 1.1]{BGRW17}. Note also that reduced reflection decompositions in parabolic subgroups and in the whole group coincide, see \cite[Theorem 1.4]{BDSW14}. Hence the main goal of this paper is to prove Theorem \ref{th:themaintheorem1} for the case that $w$ is a quasi-Coxeter element. 

\bigskip
\noindent \textbf{Structure of the paper.} The strucuture of the paper is as follows. In Section \ref{sec:AffineLattice} we recall the construction of an affine Coxeter group starting from a crystallographic root system. We further state (and prove) some results about root lattices and affine Coxeter groups which will be useful later in the paper when we work with reflection decompositions in affine Coxeter groups. In Section \ref{sec:DualCoxeter} we introduce the notion of dual Coxeter systems. In Section \ref{sec:QuasiCoxHurwitz} we give the precise definition of a quasi-Coxeter element in a dual Coxeter system. We then collect some facts about reflection decompositions and in particular about the Hurwitz action on reflection decompositions of (quasi-)Coxeter elements. By definition, an element in a Coxeter group $W$ of rank $n$ is a quasi-Coxeter element if it has at least one reduced reflection decomposition of length $n$ such that the corresponding reflections generate $W$. In Section \ref{sec:QuasiCoxElInAffine} we will show that a minimal set of reflections generating an affine Coxeter group also yields a reduced reflection decomposition of a quasi-Coxeter element. In Section \ref{sec:GenerationAffine} we start by investigating minimal generating sets of affine Coxeter groups given by reflections. Given such a minimal generating set which contains two reflections in parallel affine hyperplanes, the corresponding hyperplanes must be adjacent. This is formulated as Theorem \ref{cor:ParWallLongRoot} (and later used in Section \ref{sec:ProofMainResult}). The rest of Section \ref{sec:GenerationAffine} is devoted to the proof of Theorem \ref{cor:ParWallLongRoot}. In Section \ref{sec:ProofMainResult} we finally prove the Main Theorem \ref{th:themaintheorem1}. We do this by projecting the reflection decompositions of a quasi-Coxeter element in an affine Coxeter group to the underlying finite Coxeter group. Then we investigate the Hurwitz action on these projected decompositions as well as on their fibres. We will make heavy use of the results shown before, in particular of the results from Sections \ref{sec:QuasiCoxHurwitz}, \ref{sec:QuasiCoxElInAffine} and \ref{sec:GenerationAffine}.

\bigskip
\noindent \textbf{Acknowledgments.} 
The author wishes to thank Barbara Baumeister, Thomas Gobet, Joel Lewis, Stefan Witzel and Sophiane Yahiatene for helpful discussions and comments. The author also wishes to thank two anonymous readers for helpful comments.

\section{Root lattices and affine Coxeter groups} \label{sec:AffineLattice}
Let $V$ be a euclidean vector space with positive definite symmetric bilinear form $(- \mid -)$. Let $\Phi$ be a crystallographic root system in $V$ with simple system $\Delta$ (see \cite{Hum90} for the necessary background on root systems). The set 
$$\Phi^{\vee}:=\{ \alpha^{\vee} \mid \alpha \in \Phi\},$$
where $\alpha^{\vee}:=\frac{2 \alpha}{(\alpha \mid \alpha)}$, is again a crystallographic root system in $V$ with simple system $\Delta^{\vee}:=\{ \alpha^{\vee} \mid \alpha \in \Delta \}$. The root system $\Phi^{\vee}$ is called the \defn{dual root system} and its elements are called \defn{coroots}. (Note that in \cite{Bou02} the dual root system is defined in the dual space $V^*$ since they work with a not necessarily euclidean vector space.) For a set of roots $R \subseteq \Phi$ we put $W_R:=\langle s_{\alpha} \mid \alpha \in R \rangle$. We denote the highest root of $\Phi$ with respect to $\Delta$ by $\widetilde{\alpha}$.

\begin{Definition}
For a set of vectors $\Phi\subseteq V$ we set $L(\Phi):=\spanz(\Phi)$. If $\Phi$ is a root system, then $L(\Phi)$ is a lattice, called the \defn{root lattice}.  If $\Phi$ is a crystallographic root system, then $L(\Phi)$ is an integral lattice. In the latter case we call $L(\Phi^{\vee})$ the \defn{coroot lattice}.
\end{Definition}

There is a connection between generating sets of $W_{\Phi}$ and bases of the lattices $L(\Phi)$ and $L(\Phi^{\vee})$, respectively.
\begin{theorem} \label{PropVoigt}
\cite[Theorem 1.1]{BW18}
Let $\Phi$ be a crystallographic root system, 
$\Phi' \subseteq \Phi$ be a root subsystem, $R:= \{ \beta_1 , \ldots , \beta_k \} \subseteq  \Phi'$ be a non-empty set of roots. The following statements are equivalent:
 \begin{enumerate}
  \item[(a)] The root subsystem $\Phi'$ is the smallest root subsystem of $\Phi$ containing $R$ (i.e., the intersection of all root subsystems containing $R$);
  \item[(b)] $W_{\Phi'}= W_R$;
  \item[(c)] $L(\Phi') = L(R)$ and $L((\Phi')^{\vee})= L(R^{\vee})$.
 \end{enumerate}
\end{theorem}

\begin{Definition}
Let $\Phi$ be a crystallographic root system in $V$. The \defn{weight lattice} $P(\Phi)$ of $\Phi$ is defined by
$$P(\Phi):=\{ x \in V \mid (x \mid \alpha^{\vee}) \in \ZZ~\forall \alpha \in \Phi \}.$$
Similarly, the \defn{coweight lattice} $P(\Phi^{\vee})$ is defined by
$$P(\Phi^{\vee}):=\{ x \in V \mid (x \mid \alpha) \in \ZZ~\forall \alpha \in \Phi \}.$$
By \cite[VI, 9, Prop. 26]{Bou02} $P(\Phi)$ (resp. $P(\Phi^{\vee})$) is again a lattice containing
$L(\Phi)$ (resp. $L(\Phi^{\vee})$). 
\end{Definition}

\medskip
\noindent We summarize the definition of an affine Coxeter group and some of its elementary properties. For details and proofs we refer to \cite[Ch. 4]{Hum90}.

Throughout the rest of this section we fix a euclidean vector space with positive definite symmetric bilinear form $(- \mid -)$ and a crystallographic root system $\Phi$ in $V$. For each $\alpha \in \Phi$ and each $k \in \ZZ$, the set 
$$ H_{\alpha, k}:= \{v \in V \mid (v \mid \alpha)=k \}$$
defines an affine hyperplane. We define the \defn{affine reflection} $s_{\alpha,k}$ in $H_{\alpha, k}$ by 
$$ s_{\alpha, k}: V \rightarrow V, v \mapsto v -((v \mid \alpha) -k) \alpha^{\vee}.$$
Then $s_{\alpha, k}$ fixes $H_{\alpha, k}$ pointwise and sends $0$ to $k \alpha^{\vee}$. Moreover one has $s_{\alpha, k}= s_{- \alpha, -k}$.
\begin{Proposition} \label{prop:ConjInAffine}
\cite[Proposition 4.1]{Hum90}
Let $w \in W_{\Phi}$, $\alpha \in \Phi$ and $k \in \ZZ$. 
\begin{enumerate}
\item[(a)] $w H_{\alpha, k}= H_{w(\alpha), k}$;
\item[(b)] $w s_{\alpha, k} w^{-1} = s_{w(\alpha), k}$.
\end{enumerate}
\end{Proposition}

\medskip
\noindent For each $x \in V$ we define the translation in $x$ by 
$$\TR(x): V \rightarrow V, v \mapsto v+x.$$
The set $\Tr(V):=\{ \TR(x) \mid x \in V \}$ of all translations by elements of $V$ is a group. The following result is straightforward to check. 

\begin{Lemma} \label{le:CalcAffine}
Let $\alpha \in \Phi$ and $k,\ell \in \ZZ$.

\noindent
\begin{enumerate}
\item[(a)] $s_{\alpha, k }= \TR(k \alpha^{\vee}) s_{\alpha}= s_{\alpha} \TR(-k \alpha^{\vee}).$ In particular one has $s_{\alpha, 0 }= s_{\alpha}$;
\item[(b)] $s_{\alpha} s_{\alpha, 1}= \TR(- \alpha^{\vee})$;
\item[(c)] $s_{\alpha, k} s_{\alpha, \ell}= \TR((k-\ell)\alpha^{\vee})$.
\end{enumerate}
\end{Lemma}


\medskip
\noindent Let $\Aff(V)$ be the semidirect product of the general linear group $\GL(V)$ and $\Tr(V)$ which we call the \defn{affine group} of $V$.

\begin{Definition}
The group 
$$W_{a, \Phi} := \langle s_{\alpha, k} \mid \alpha \in \Phi, k \in \ZZ \rangle \leq \Aff(V)$$
is called \defn{affine Weyl group} associated to $\Phi$. We sometimes omit the subscript $\Phi$ if it is clear from the context and denote $W_{a,\Phi}$ by $W_a$.
\end{Definition}

\begin{theorem} \label{thm:StructureAffineWeylGroup}
\cite[Proposition 4.2, Theorem 4.6]{Hum90}
Let $\Phi$ be a crystallographic root system with simple system $\Delta$.
\begin{enumerate}
\item[(a)] The group $W_a=W_{a, \Phi}$ is the semidirect product of the finite Coxeter group $W_{\Phi}$ and the group 
$$\Tr(\Phi^\vee):= \{ \TR(\alpha) \mid \alpha \in L(\Phi^{\vee}) \},$$
which we identify with $L(\Phi^{\vee})$.
\item[(b)] $(W_a, S_a)$ is a Coxeter system, where
$$S_a= S_{a, \Delta}:= \{s_{\alpha} \mid \alpha \in \Delta \} \cup \{ s_{\widetilde{\alpha},1} \},$$
and the set of reflections for $(W_a, S_a)$ is given by the set of affine reflections, that is
$$
T_a := \{ s_{\alpha, k} \mid \alpha \in \Phi, ~k \in \ZZ\}.
$$
\end{enumerate}
\end{theorem}
Therefore the affine Weyl group $W_a$ is also called \defn{affine Coxeter group} and the pair $(W_a, S_a)$ is called \defn{affine Coxeter system}.

\medskip


Theorem \ref{thm:StructureAffineWeylGroup} provides the following normal form for elements in an affine Coxeter group.

\begin{Lemma} \label{le:McC}
For each element $w \in W_a = W_{a, \Phi}$ there is a unique factorization $w=w_0 \TR(\lambda)$ with $w_0 \in W_{\Phi}$ and $\lambda \in L(\Phi^{\vee})$.
\end{Lemma}

\begin{Lemma} \label{le:CalcAffine2}
Let $\alpha \in \Phi$, $\lambda \in P(\Phi^{\vee})$, $w \in W_{\Phi}$ and $k \in \ZZ$. Then the following holds:

\begin{enumerate}
\item[(a)] $\TR(\lambda) H_{\alpha, k} = H_{\alpha, k+(\lambda \mid \alpha)}$;
\item[(b)] $\TR(\lambda) s_{\alpha, k} \TR(-\lambda) = s_{\alpha, k +(\lambda \mid \alpha)}$, that is, $P(\Phi^{\vee})$ normalizes $W_a$;
\item[(c)] $w \TR(\alpha^{\vee}) w^{-1}= \TR(w(\alpha)^{\vee})$.
\end{enumerate}
\end{Lemma}

\begin{proof}
For parts $(a)$ and $(b)$ see \cite[Prop. 4.1]{Hum90}. For part $(c)$ it is sufficient to consider $w=s_{\beta}$ for some $\beta \in \Phi$. By Proposition \ref{prop:ConjInAffine} and Lemma \ref{le:CalcAffine} we obtain 
$$s_{\beta} \TR(\alpha^{\vee}) s_{\beta} = s_{\beta} s_{\alpha, 1} s_{\alpha} s_{\beta} = s_{s_{\beta}(\alpha),1}  s_{s_{\beta}(\alpha)} = \TR(s_{\beta}(\alpha)^{\vee}).$$
\end{proof}

\begin{Lemma} \label{le:ActionOnCoweight}
The group $W_{\Phi}$ acts on $P(\Phi^{\vee})$. 
\end{Lemma}

\begin{proof}
Let $\lambda \in P(\Phi^{\vee})$, thus $(\lambda \mid \alpha) \in \ZZ$ for all $\alpha \in \Phi$. Since $W_{\Phi}$ is a subgroup of the orthogonal group $O(V)$, we obtain
$$
(w(\lambda) \mid \alpha) = (\lambda \mid w^{-1}(\alpha)) \in \ZZ ~\text{for all }\alpha \in \Phi,
$$
because $w^{-1}(\alpha) \in \Phi$. Hence $w(\lambda) \in P(\Phi^{\vee})$. 
\end{proof}

\medskip
\noindent Lemma \ref{le:McC} provides a normal form for each element in an affine Coxeter group. Given a reflection decomposition of an element in $W_a$, the following lemma tells us how this normal form is achieved.

\begin{Lemma} \label{le:AffineFac}
For $\beta_i \in \Phi$ and $k_i \in \mathbb{Z}$ $(1 \leq i \leq m)$ we have 
\begin{align*}
s_{\beta_1, k_1} \cdots s_{\beta_m, k_m}  = 
s_{\beta_1} \cdots s_{\beta_m}
\TR(\sum_{i=1}^m -k_i s_{\beta_m} \cdots s_{\beta_{i+1}}(\beta_i)^{\vee}).
\end{align*}
\end{Lemma}

\begin{proof}
By Lemma \ref{le:CalcAffine} we have for $\alpha \in \Phi, k \in \ZZ$ that
\begin{align}\label{equ:AffRef}
s_{\alpha,k}=\TR(k\alpha^{\vee})s_{\alpha}=s_{\alpha} \TR(-k \alpha^{\vee}) 
\end{align} 
and by Proposition \ref{prop:ConjInAffine} that
\begin{align} \label{equ:ActionRef}
w s_{\alpha, k} w^{-1} &= s_{w(\alpha), k}~ \text{for all } w \in W_{\Phi}.
\end{align}
We show the assertion by induction on $m$. It is clear for $m=1$. By induction it follows
\begin{align*}
s_{\beta_1, k_1} \cdots s_{\beta_m, k_m} = 
s_{\beta_1,k_1} s_{\beta_2} \cdots s_{\beta_m}
\TR(\sum_{i=2}^m -k_i s_{\beta_m} \cdots s_{\beta_{i+1}}(\beta_i)^{\vee}).
\end{align*}
Put $w:=s_{\beta_m} \cdots s_{\beta_2}$. Then 
\begin{align*}
s_{\beta_1,k_1} s_{\beta_2} \cdots s_{\beta_m}  \stackrel{\phantom{(\ref{equ:ActionRef})}}{=} & s_{\beta_1, k_1} w^{-1}\\
{} \stackrel{(\ref{equ:ActionRef})}{=} & w^{-1} s_{w(\beta_1), k_1}\\
{} \stackrel{(\ref{equ:AffRef})}{=} & w^{-1} s_{w(\beta_1)} \TR(-k_1 w(\beta_1)^{\vee})\\
{} \stackrel{\phantom{(\ref{equ:ActionRef})}}{=} & s_{\beta_1} w^{-1} \TR(-k_1 w(\beta_1)^{\vee})\\
{} \stackrel{\phantom{(\ref{equ:ActionRef})}}{=} & s_{\beta_1}s_{\beta_2} \cdots s_{\beta_m} \TR(-k_1 s_{\beta_m} \cdots s_{\beta_2}(\beta_1)^{\vee}),
\end{align*}
which yields the assertion.
\end{proof}

\begin{Lemma} \label{le:AffConj}
For $\alpha, \beta \in \Phi$ and $k,\ell \in \ZZ$ we have 
$$s_{\alpha, k} s_{\beta, \ell} s_{\alpha, k} = s_{s_{\alpha}(\beta), \ell- k \frac{2(\alpha \mid \beta)}{(\alpha \mid \alpha)} }.$$
\end{Lemma}

\begin{proof}
By Lemma \ref{le:AffineFac} we have 
$$s_{\alpha, k} s_{\beta, \ell} s_{\alpha, k} = s_{s_{\alpha}(\beta)} \TR(-k s_{\alpha}s_{\beta}(\alpha)^{\vee} - \ell s_{\alpha}(\beta)^{\vee} -k \alpha^{\vee}).$$
One has
\begin{align*}
s_{\alpha}s_{\beta}(\alpha)^{\vee} + \alpha^{\vee} & = 
\frac{2 s_{\alpha}s_{\beta}(\alpha)}{(\alpha \mid \alpha)} + \frac{2 \alpha}{(\alpha \mid \alpha)}\\
& = - \frac{2}{(\alpha \mid \alpha)} \cdot \frac{2(\alpha \mid \beta)}{(\beta \mid \beta)} s_{\alpha}(\beta)\\
& = - \frac{2(\alpha \mid \beta)}{(\alpha \mid \alpha)} s_{\alpha}(\beta)^{\vee},
\end{align*}
which yields the assertion.
\end{proof}

\medskip


\section{Dual Coxeter systems} \label{sec:DualCoxeter}

Let $(W,T)$ be a pair consisting of a group $W$ and a generating subset $T$ of $W$. In the sense of \cite{Bes03} we call $(W,T)$ a \defn{dual Coxeter system} of finite rank $n$ if there is a subset $S \subseteq T$ with $|S| = n$ such that 
$(W,S)$ is a Coxeter system, and $T = \big\{ wsw^{-1} \mid w \in W, s \in S \big\}$ 
is the set of reflections for the Coxeter system $(W,S)$.
We then call~$(W,S)$ a \defn{simple system} for $(W,T)$. If $S'\subseteq T$ is such that $(W,S')$ is a Coxeter system, then $\big\{ wsw^{-1} \mid w \in W, s \in S' \big\}=T$ (see \cite[Lemma 3.7]{BMMN02}). Hence a set $S'\subseteq T$ is a simple system for $(W,T)$ if and only if $(W,S')$ is a Coxeter system. The \defn{rank} of $(W,T)$ is defined as $|S|$ for a simple system $S \subseteq T$. This is well-defined by \cite[Theorem 3.8]{BMMN02}.

Let $W'$ be a reflection subgroup of $W$. Then $(W', W' \cap T)$ is again a dual Coxeter system by \cite{Dye90}. The reflection subgroup generated by $\{s_1,\ldots,s_m\} \subseteq T$ is called a \defn{parabolic subgroup} for $(W,T)$ if there is a simple system~$S = \{ s_1,\ldots,s_n \}$ for $(W,T)$
with $m \leq n$. This definition differs from the usual notion of a parabolic subgroup generated by a conjugate of a subset of a fixed simple system $S$ (see \cite[Section 1.10]{Hum90}) which we call a \defn{classical parabolic subgroup} for $(W,S)$. Obviously  a classical parabolic subgroup is a parabolic subgroup as defined here. But the two definitions are not equivalent in general (see \cite[Example 2.2]{Gob}). By \cite[Proposition 4.6]{BGRW17} the notions of parabolic subgroup and classical parabolic subgroup coincide in finite and affine Coxeter groups.

A dual Coxeter system $(W,T)$ is called \defn{finite} (resp. \defn{affine}) if $(W,S)$ is finite (resp. affine) for one (equivalently each) simple system $S \subseteq T$. This is well-defined by \cite{BMMN02, FHM06}. A dual Coxeter system $(W,T)$ is called \defn{irreducible} if $(W,S)$ is irreducible for each simple system $S \subseteq T$.

\section{Quasi-Coxeter elements, Hurwitz action and the absolute order} \label{sec:QuasiCoxHurwitz}

In this section we give the definition of a quasi-Coxeter element in a dual Coxeter system and collect some facts about reflection decompositions, in particular about the Hurwitz action on reflection decompositions of (quasi-)Coxeter elements.

\begin{Definition} \label{def:AbsoluteOrder}
Let $(W,T)$ be a dual Coxeter system. The partial order on $W$ defined by
$$u\leq_T v\text{ if and only if }\ell_T(u)+\ell_T(u^{-1}v)=\ell_T(v)$$
for $u,v \in W$ is called \defn{absolute order}.
\end{Definition}

In other words, we have $u\leq_T v$ if and only if there exists an element in $\Red_T(u)$ which is a prefix of an element in $\Red_T(v)$. To determine whether a reflection decomposition is reduced, Carter gave a geometric criterion for finite Coxeter groups.

\begin{Lemma} \label{lem:Carter}
\cite[Lemma 3]{Carter} Let $\Phi$ be a root system and $\alpha_1, \ldots, \alpha_k \in \Phi$. Then $s_{\alpha_1} \cdots s_{\alpha_k}$ is a reduced reflection decomposition if and only if $\alpha_1, \ldots, \alpha_k$ are linearly independent. 
\end{Lemma}

\begin{Lemma} \label{le:EachRefBelow}
Let $(W,T)$ be a finite dual Coxeter system of rank $n$ and $w \in W$ with $\ell_T(w)=n$. Then $t \leq_T w$ for all $t \in T$.
\end{Lemma}

\begin{proof}
By Lemma \ref{lem:Carter} we have $\ell_T(tw)=n-1$. Hence there exists a reduced reflection decomposition for $w$ which begins with $t$. 
\end{proof}

\begin{Definition}\label{def:coxeter} 
Let $(W, T)$ be a dual Coxeter system of rank $n$.
\begin{enumerate}

\item[(a)] An element $c\in W$ is called a \defn{Coxeter element} if there exists a simple system $S=\{s_1,\dots, s_n\}$ for $(W,T)$ such that $c=s_1 \cdots s_n$. An element $w\in W$ is called a \defn{parabolic Coxeter element} if there exists a simple system $S=\{s_1,\dots, s_n\}$ for $(W,T)$ such that $w=s_1\cdots s_m$ for some $m\leq n$. The element~$w$ is moreover called a \defn{standard parabolic Coxeter element} for the Coxeter system $(W,S)$.
\item[(b)] An element $w \in W$ is called a \defn{quasi-Coxeter element} for $(W,T)$ if there exists a reduced reflection decomposition $(t_1,\dots ,t_n)\in\Red_T(w)$ such that $W= \langle t_1 , \ldots , t_n \rangle$. An element $w \in W$ is called a \defn{parabolic quasi-Coxeter element} for $(W,T)$ if there is a simple system $S= \{ s_1 , \ldots , s_n\}$ and $(t_1 , \dots , t_m)\in\Red_T(w)$ such that $\langle t_1 , \ldots , t_m \rangle = \langle s_1 , \ldots , s_m \rangle$ for some $m \leq n$.
\end{enumerate}
\end{Definition}

Let us first point out that quasi-Coxeter elements in affine Coxeter groups indeed extend the class of Coxeter elements. In Example \ref{ex:CoxNotQuasiCox} we will provide an example of a quasi-Coxeter element in an affine Coxeter group which is not a Coxeter element. 

Furthermore we want to note that the given definition of a quasi-Coxeter element appears naturally in finite (dual) Coxeter systems. Here the rank of the Coxeter system is a natural bound for the reflection length. In affine Coxeter systems we still have bounded reflection length as shown in \cite{McCP11}. But this bound exceeds the rank of the affine Coxeter system. Therefore the definition of a quasi-Coxeter element might be extended to longer reduced reflection decompositions (and even in such a way that Theorem \ref{th:themaintheorem1} is still valid for this extended definition). We will adress this topic at the end of the paper in Remark \ref{rem:end} and Example \ref{ex:Thomas}.

\medskip
\noindent Let us collect some results about the Hurwitz action and quasi-Coxeter elements.

\begin{theorem}
\label{th:maintheorem1}
  \cite[Theorem 1.3]{BDSW14} Let $(W,T)$ be a dual Coxeter system of rank~$n$ and 
  let~$w = s_1\cdots s_m$ be a parabolic Coxeter element in~$W$.
  The Hurwitz action on $\Red_T(w)$ is transitive, that is, for each $(t_1,\ldots,t_m) \in \Red_T(w)$ there is a braid $\sigma \in \mathcal{B}_m$ such that
  $$\sigma (t_1,\ldots,t_m) = (s_1,\ldots,s_m).$$
\end{theorem}

\begin{example}  \label{ex:CoxNotQuasiCox}
Let $(W,T)$ be an affine dual Coxeter system of type $\widetilde{D}_n$, that is, $W=W_{a,\Phi}$ with $\Phi$ a root system of type $D_n$. Consider the Coxeter element $c:= s_{\alpha_1} \cdots s_{\alpha_{n}}s_{\widetilde{\alpha},1}$, where $\alpha_1, \ldots, \alpha_n$ are simple roots for $\Phi$ and $\widetilde{\alpha}$ is the corresponding highest root. The element $c':= s_{\alpha_1} \cdots s_{\alpha_{n}}$ is a Coxeter element for $W_{\Phi}$. Using Theorem \ref{th:maintheorem1} and the fact that $t \leq_T c'$ for all $t \in T$, we find $\beta_1, \ldots , \beta_{n-1} \in \Phi$ and a braid $\sigma \in \mathcal{B}_n$ such that
$$
\sigma(s_{\alpha_1}, \ldots, s_{\alpha_{n}})=(s_{\beta_1}, \ldots, s_{\beta_{n-1}}, s_{\widetilde{\alpha}}).
$$
Consider the natural epimorphism $p: W_{a,\Phi} \rightarrow W_{\Phi}, ~s_{\alpha, k} \mapsto s_{\alpha}$. Then one has
$$p(c) = s_{\beta_1} \cdots s_{\beta_{n-1}} \leq_T c'.$$
Thus $p(c)$ is a parabolic Coxeter element by \cite[Corollary 3.6]{DG}. For all $x \in W_{a, \Phi}$ we have $p(xcx^{-1})= p(x) p(c) p(x)^{-1}$. Since all Coxeter elements in $W_{a,\Phi}$ are conjugate (note that the $\widetilde{D}_n$-diagram is a tree), we see that the projection of an arbitrary Coxeter element in $W_{a, \Phi}$ under $p$ is always a parabolic Coxeter element in $W_{\Phi}$. If we substitute the Coxeter element $c'=s_{\alpha_1} \cdots s_{\alpha_{n}}=s_{\beta_1} \cdots s_{\beta_{n-1}} s_{\widetilde{\alpha}}$ by an arbitrary quasi-Coxeter element $w'$ in $W_{\Phi}$, then $w' s_{\widetilde{\alpha},1} \in W_{a, \Phi}$ is quasi-Coxeter and $p(w' s_{\widetilde{\alpha},1}) \in W_{\Phi}$ is a parabolic quasi-Coxeter element by \cite[Corollary 6.11]{BGRW17}. In \cite[Example 2.4]{BGRW17} we provided an example of a parabolic quasi-Coxeter element in type $D_4$ which is not a parabolic Coxeter element. This shows that $p(w' s_{\widetilde{\alpha},1})$ needs not to be a parabolic Coxeter element in $W_{\Phi}$.
\end{example}

\begin{theorem} \label{thm:RedParabolic}
\label{th:maintheorem2} \cite[Theorem 1.4]{BDSW14} Let $(W,T)$ be a dual Coxeter system and let $W'$ be a parabolic subgroup of $W$. 
Then for any $w\in W'$, 
  $$\Red_T(w) = \Red_{T'}(w),$$
  where $T' = W' \cap T$ is the set of reflections in $W'$.
\end{theorem}

\begin{theorem}
\label{th:maintheorem3}
\cite[Theorem 1.1]{BGRW17} Let $(W,T)$ be a finite dual Coxeter system and let~$w \in W$. The Hurwitz action on $\Red_T(w)$ is transitive if and only if $w$ is a parabolic quasi-Coxeter element for $(W,T)$.
\end{theorem}

In general, Theorem \ref{th:maintheorem3} does not apply to affine (or other infinite) Coxeter groups. Hubery and Krause gave an example \cite[Example 5.7]{HK13} of an element in an affine Coxeter group such that the Hurwitz action is transitive on its set of reduced reflection decompositions, but $w$ is not a parabolic Coxeter element.

For finite Coxeter groups it has been shown by Lewis and Reiner that reflection decompositions can be reduced by using the Hurwitz action. 
\begin{theorem} \label{thm:LR2}
\cite[Corollary 1.4]{LR16} Let $(W,T)$ be a finite dual Coxeter system and $w \in W$ with $\ell_T(w)=n$. Then every decomposition of $w$ into $m$ reflections
lies in the Hurwitz orbit of some  $(t_1, \ldots , t_m)$ such that
\begin{align*}
t_1  =  t_2,~ t_3  =  t_4, \ldots, ~t_{m-n-1} = t_{m-n},
\end{align*}
and $(t_{m-n+1}, \ldots , t_m) \in \Red_T(w)$.
\end{theorem}

\noindent We use the notation 
$$ (t_1 , \ldots , t_n) \sim (t_1', \ldots, t_n')$$
to indicate that both $n$-tuples are in the same orbit under the Hurwitz action. For the next statement see also the proof of \cite[Theorem 1.1]{LR16}
\begin{Lemma} \label{le:Reduction}
Let $(W,T)$ be a dual Coxeter system and $t_1, \ldots, t_{n}, t \in T$. Then
$$(t_1, \ldots, t_{n},t,t) \sim (t_1, \ldots, t_{n}, xtx^{-1},xtx^{-1})$$
for each $x \in \langle t_1, \ldots, t_{n} \rangle$. 
\end{Lemma}

\begin{proof}
We have
\begin{align*}
(t_1, \ldots, t_{n},t,t) & \sim (t_1, \ldots, \hat{t_i}, t_i t_{i+1} t_i \ldots , t_i t_{n} t_i, t_i t t_i,t_i t t_i, t_i)\\
{} & \sim (t_1, \ldots, \hat{t_i}, t_i t_{i+1} t_i \ldots , t_i t_{n} t_i, t_i, t_i t t_i,t_i t t_i)\\
{} & \sim (t_1, \ldots, t_{n},t_i t t_i , t_i t t_i),
\end{align*}
where the entry $\hat{t_i}$ is omitted. 
\end{proof}

\section{Quasi-Coxeter elements in affine Coxeter groups} \label{sec:QuasiCoxElInAffine}
The aim of this section is to prove the follwing statement about quasi-Coxeter elements in affine Coxeter groups.
\begin{Proposition} \label{prop:CharacAffineQuasicCox}
An element $w \in W$ is quasi-Coxeter for an affine dual Coxeter system $(W,T)$ of rank $n+1$ if and only if there exists a reflection decomposition $w= s_{\gamma_1, \ell_1} \cdots s_{\gamma_{n+1}, \ell_{n+1}}$ $(\gamma_i \in \Phi,~\ell_i \in \ZZ)$ such that $W= \langle s_{\gamma_1, \ell_1}, \ldots , s_{\gamma_{n+1}, \ell_{n+1}} \rangle$. 
\end{Proposition}

\medskip
Fix a finite crystallographic Coxeter system $(W_0,S_0)$ of rank $n$ with set of reflections $T_0$ and associated crystallographic root system $\Phi$ in a euclidean vector space $V$. The group $W_0$ is a proper parabolic subgroup of the associated affine Coxeter group $W_a=W_{a, \Phi}$. There is an epimorphism
$$p: W_a \rightarrow W_0, s_{\alpha, k} \mapsto s_{\alpha}.$$
In Theorem \ref{thm:StructureAffineWeylGroup} we saw how to obtain a simple system $S_a$ for $W_a$. Let $T_a$ be the set of reflections for $(W_a,S_a)$. In particular $(W_a,T_a)$ is an affine dual Coxeter system of rank $n+1$. We fix this system until the end of this section. Since $T_a=\{ s_{\alpha, k} \mid \alpha \in \Phi, ~k \in \ZZ \}$, we have $p(T_a)=T_0$. 


\medskip

\medskip


\begin{Proposition} \label{prop:SameRootsOrbit}
Let $x \in W_a$ with $\ell_{T_a}(x)=m \geq n+1$ and $(s_{\beta_1,k_1}, \ldots ,s_{\beta_m, k_m}) \in \Red_{T_a}(x)$. Then there exist $\beta_i' \in \Phi, k_i' \in \ZZ$ $(1 \leq i \leq m)$ such that
$$ (s_{\beta_1,k_1}, \ldots , s_{\beta_m, k_m}) \sim 
(s_{\beta_1',k_1'}, \ldots ,s_{\beta_{m-1}',k_{m-1}'}, s_{\beta_{m-1}', k_m'}).$$
\end{Proposition}

\begin{proof}
By Lemma \ref{lem:Carter} the reflection decomposition $s_{\beta_1} \cdots s_{\beta_m}$ can not be reduced in the underlying finite group $W_0$, because the roots $\beta_1, \ldots, \beta_m$ can not be linearly independent. Since $W_0$ is a parabolic subgroup of $W_a$, this reflection decomposition can also not be reduced in $W_a$ by Theorem \ref{thm:RedParabolic}. Hence we can apply Theorem \ref{thm:LR2} to obtain
$$ (s_{\beta_1}, \ldots , s_{\beta_m}) \sim 
(s_{\beta_1'}, \ldots ,s_{\beta_{m-1}'} ,s_{\beta_{m-1}'})$$
for roots $\beta_1', \ldots, \beta_{m-1}' \in \Phi$. This Hurwitz equivalence is given by some braid $\tau$. Applying the braid $\tau$ to the reflection decomposition $(s_{\beta_1,k_1}, \ldots , s_{\beta_m, k_m})$ yields the assertion.
\end{proof}

\medskip

\noindent Let $w= s_{\gamma_1', k_1} \cdots s_{\gamma_{n+1}', k_{n+1}}$ ($\gamma_i' \in \Phi, k_i \in \ZZ$) be a reflection decomposition of an element $w \in W_a$ such that $W_a= \langle s_{\gamma_1', k_1}, \ldots , s_{\gamma_{n+1}', k_{n+1}} \rangle$. Note that $W_a$ cannot be generated by fewer than $n+1$ reflections \cite[Proposition 2.1]{BGRW17}. Using Theorem \ref{thm:LR2}, we obtain
\begin{align} \label{equ:ParabolicQuasi}
(s_{\gamma_1'}, \ldots, s_{\gamma_{n+1}'}) \sim
(s_{\gamma_1}, \ldots, s_{\gamma_{n-1}}, s_{\gamma_{n}}, s_{\gamma_{n}})
\end{align}
for roots $\gamma_1, \ldots , \gamma_n \in \Phi$. Since $p(W_a)=W_0$, we have 
$$W_0= \langle s_{\gamma_1}, \ldots, s_{\gamma_{n-1}}, s_{\gamma_{n}} \rangle.$$
Therefore the element $v:=s_{\gamma_1} \cdots s_{\gamma_{n-1}} s_{\gamma_{n}}$ is quasi-Coxeter in $W_0$. Furthermore we have 
\begin{align} \label{equ:ThisIsQuasiCox}
v' :=p(w)= s_{\gamma_1} \cdots s_{\gamma_{n-1}} \leq_{T} s_{\gamma_1} \cdots s_{\gamma_{n-1}} s_{\gamma_{n}}=v.
\end{align}
Then \cite[Corollary 6.11]{BGRW17} tells us that $v'$ is a parabolic quasi-Coxeter element in $W_0$ with $\ell_{T_0}(v')=n-1$. Let $W_0':=\langle s_{\gamma_1}, \ldots , s_{\gamma_{n-1}} \rangle$ be the corresponding parabolic subgroup. As described in (the proof of) Proposition \ref{prop:SameRootsOrbit} we have
\begin{align} \label{equ:AffineQuasi}
(s_{\gamma_1', k_1}, \ldots, s_{\gamma_{n+1}', k_{n+1}}) \sim 
(s_{\gamma_1, \ell_1}, \ldots, s_{\gamma_{n-1}, \ell_{n-1}}, s_{\gamma_{n}, \ell_{n}}, s_{\gamma_{n}, \ell_{n+1}})
\end{align}
for $\ell_1, \ldots, \ell_{n+1} \in \ZZ$. Note that we can assume the roots $\gamma_i$ to be positive since $s_{\alpha, k}= s_{-\alpha, -k}$. We conclude
\begin{align} \label{equ:GenAffineGroup}
W_a = \langle s_{\gamma_1, \ell_1}, \ldots, s_{\gamma_{n-1}, \ell_{n-1}}, s_{\gamma_{n}, \ell_{n}}, s_{\gamma_{n}, \ell_{n+1}} \rangle
\end{align}
and $\ell_{n} \neq \ell_{n+1}$, because otherwise $W_a$ is generated by fewer than $n+1$ reflections.

\begin{Lemma} \label{le:LinIndCoroots}
For a crystallographic root system $\Phi$ and a set of roots $\{ \beta_1, \ldots, \beta_n \} \subseteq \Phi$, one has
$$s_{\beta_n} \cdots s_{\beta_2} (\beta_1)^{\vee} \in \spanz(\beta_1^{\vee}, \ldots, \beta_n^{\vee}).$$
\end{Lemma}

\begin{proof}
One has
\begin{align*}
s_{\beta_2}(\beta_1)^{\vee} & = \frac{2s_{\beta_2}(\beta_1)}{(\beta_1 \mid \beta_1)}\\
{} & = \frac{2}{(\beta_1 \mid \beta_1)} (\beta_1 - (\beta_1 \mid \beta_2^{\vee}) \beta_2)\\
{} & = \beta_1^{\vee} - \underbrace{\frac{2(\beta_1 \mid \beta_2)}{(\beta_1 \mid \beta_1)}}_{\in \ZZ} \beta_2^{\vee}.
\end{align*}
The general assertion follows by induction.
\end{proof}

\medskip
\noindent Let us point out that all reflection decompositions of a fixed element have same parity.
\begin{Proposition} \label{prop:RefLengthNonRed}
Let $(W,S)$ be a Coxeter system with reflections $T$. Let $w \in W$ and $w=t_1 \cdots t_k = r_1 \cdots r_{\ell}$ be two reflection decompositions for $w$. Then $k$ and $\ell$ differ by a multiple of $2$.
\end{Proposition}

\begin{proof}
Consider the geometric representation $\sigma: W \rightarrow \GL(V)$ of $W$ as given in \cite{Hum90}. The reflections have determinant $-1$ with respect to this representation. Consider the reflecion decompositions $w=t_1 \cdots t_k$ and $w = r_1 \cdots r_{\ell}$ under the sign representation $w \mapsto \det(\sigma(w))$. Hence $(-1)^k = (-1)^{\ell}$, which yields the assertion. 
\end{proof}


\medskip
\begin{proof}[Proof of Proposition \ref{prop:CharacAffineQuasicCox}]
We prove the assertion for our fixed system $(W_a, T_a)$. The forward direction is clear by the definition of a quasi-Coxeter element. For the other direction it remains to show that $s_{\gamma_1, \ell_1} \cdots s_{\gamma_{n+1}, \ell_{n+1}}$ is a reduced reflection decomposition. By what we have observed before (in particular in (\ref{equ:AffineQuasi})), we can assume that $\gamma_n = \gamma_{n+1}$. In particular we have 
\begin{align} \label{equ:GroupIsdochGenerated}
W_0= \langle s_{\gamma_1}, \ldots, s_{\gamma_n} \rangle.
\end{align}

We first claim, that under the map $p: W_a \rightarrow W_0, ~s_{\alpha, k} \mapsto s_{\alpha}$, the element $w$ is mapped to an element in $W_0$ of absolute length $n-1$. Since $\gamma_n = \gamma_{n+1}$, we have $\ell_{T_0}(p(w)) \leq n-1$. Assume that $\ell_{T_0}(p(w))<n-1$. Then we can apply Theorem \ref{thm:LR2} to obtain 
\begin{align} \label{equ:BraidEqual}
(s_{\gamma_1}, \ldots , s_{\gamma_{n-1}}) \sim (s_{\gamma_1'}, \ldots , s_{\gamma_{n-3}'}, s_{\gamma_{n-2}'}, s_{\gamma_{n-2}'}).
\end{align}
Hence we obtain 
\begin{align*}
W_0 = p(W_a) & \stackrel{\phantom{(\ref{equ:BraidEqual})}}{=} p(\langle s_{\gamma_1, \ell_1}, \ldots ,s_{\gamma_{n-1}, \ell_{n-1}},s_{\gamma_{n}, \ell_{n}}, s_{\gamma_{n}, \ell_{n+1}} \rangle)\\
{} & \stackrel{(\ref{equ:BraidEqual})}{=} p(\langle s_{\gamma_1', \ell_1'}, \ldots ,s_{\gamma_{n-3}', \ell_{n-3}'}, s_{\gamma_{n-2}', \ell_{n-2}'}, s_{\gamma_{n-2}', \ell_{n-1}'}, s_{\gamma_{n}, \ell_{n}}, s_{\gamma_{n}, \ell_{n+1}} \rangle)\\
{} & \stackrel{\phantom{(\ref{equ:BraidEqual})}}{=} \langle s_{\gamma_1'}, \ldots , s_{\gamma_{n-2}'}, s_{\gamma_n} \rangle,
\end{align*}
contradicting again the fact that $W_0$ can not be generated by fewer than $n$ reflections.

Thus by Proposition \ref{prop:RefLengthNonRed} there are just two possibilities remaining: $\ell_{T_a}(w)=n+1$ or $\ell_{T_a}(w)=n-1$. Assume the latter one. Let $(s_{\beta_1, m_1}, \ldots, s_{\beta_{n-1}, m_{n-1}}) \in \Red_{T_a}(w)$. Then one has
$$v'=p(w)=s_{\gamma_1} \cdots s_{\gamma_{n-1}}= s_{\beta_1} \cdots s_{\beta_{n-1}} \leq_{T_0} s_{\gamma_1} \cdots s_{\gamma_{n-1}} s_{\gamma_n}.$$
By (\ref{equ:GroupIsdochGenerated}) we have that $s_{\gamma_1} \cdots s_{\gamma_{n-1}} s_{\gamma_n}$ is a quasi-Coxeter element in $W_0$. Then \cite[Corollary 6.11]{BGRW17} implies that $v'$ is a parabolic quasi-Coxeter element in $W_0$. Therefore by Theorem \ref{th:maintheorem3} we have 
$$(s_{\gamma_1}, \ldots , s_{\gamma_{n-1}}) \sim (s_{\beta_1}, \ldots , s_{\beta_{n-1}}).$$
Hence we can assume (up to Hurwitz equivalence) that $\gamma_i = \beta_i$ for $1 \leq i \leq n-1$. Note that the integers $\ell_1, \ldots, \ell_{n+1}$ might have changed. By applying Lemma \ref{le:AffineFac} we obtain 
\begin{align*}
w & = v' \TR \left( \sum_{i=1}^{n-1} -\ell_i s_{\beta_{n-1}} \cdots s_{\beta_{i+1}}(\beta_i)^{\vee} + (\ell_{n}-\ell_{n+1})\gamma_{n}^{\vee} \right)\\
{}  & = v' \TR \left(\sum_{i=1}^{n-1} -m_i s_{\beta_{n-1}} \cdots s_{\beta_{i+1}}(\beta_i)^{\vee} \right).
\end{align*}
Hence 
$$\sum_{i=1}^{n-1} (m_i-\ell_i) s_{\beta_{n-1}} \cdots s_{\beta_{i+1}}(\beta_i)^{\vee} + (\ell_{n}-\ell_{n+1})\gamma_{n}^{\vee} = 0.$$
The roots $ \beta_1, \ldots , \beta_{n-1}, \gamma_{n} $ are linearly independent by Lemma \ref{lem:Carter}. Therefore we obtain $\ell_{n}-\ell_{n+1}=0$ using Lemma \ref{le:LinIndCoroots}, hence $s_{\gamma_{n}, \ell_{n}} = s_{\gamma_{n}, \ell_{n+1}}$, contradicting the fact that $W_a$ can not be generated by fewer than $n+1$ reflections. 
\end{proof}

\section{Generating finite and affine Coxeter groups by reflections} \label{sec:GenerationAffine}

The main goal of this section will be to prove the following statement.

\begin{theorem} \label{cor:ParWallLongRoot}
Let $\Phi$ be an irreducible crystallographic root system of rank $n$. If 
$$W_{a,\Phi} = \langle s_{\alpha_1, \ell_1}, \ldots, s_{\alpha_{n-1}, \ell_{n-1}}, s_{\alpha_{n}, \ell_{n}}, s_{\alpha_{n}, \ell_{n+1}} \rangle$$
for roots $\alpha_1, \ldots, \alpha_n \in \Phi$ and integers $\ell_1, \ldots, \ell_n, \ell_{n+1} \in \ZZ$, then $\vert \ell_{n+1} -\ell_{n} \vert =1$ and $\alpha_{n}$ is a long root.
\end{theorem}

This theorem will be essential in the next section. We will use it to show that each reduced reflection decomposition of a quasi-Coxeter element in an affine Coxeter group $W_{a, \Phi}$ induces (via the projection $p$) a generating set of the finite Coxeter group $p(W_{a, \Phi})$.

\medskip
As a generalization of \cite[Corollary 6.10]{BGRW17} we observe in \cite[Theorem 1.3]{BW18} the following.
\begin{theorem} \label{prop:CharParSub}
Let $(W,T)$ be a finite dual Coxeter system of rank $n$ and let $W'$ be a reflection subgroup of rank $n-1$. Then $W'$ is parabolic if and only if there exists $t \in T$ such that $\langle W', t \rangle =W$. If in addition $(W,T)$ is crystallographic, then $t$ is unique up to conjugation with elements in $W'$.
\end{theorem}

\medskip
For the rest of this section fix an irreducible crystallographic root system $\Phi$ of rank $n$. 

\begin{remark} \label{rem:RootRatio}
\cite[Section 2.9]{Hum90}
The set $\{ (\alpha \mid \alpha) \mid \alpha \in \Phi \}$ has at most two elements. So if it has two elements, we distinguish between \defn{long} and \defn{short} roots. The ratio of squared root lengths can only be $2$ or $3$. Hereafter we denote this ratio by $\delta$. We decompose $\Phi$ as $\Phi = \Phi_s \cup \Phi_{\ell}$, where $\Phi_s$ is the set of short roots and $\Phi_{\ell}$ is the set of long roots. If the set has just a single element, we can assume it to be $2$ and all roots are called long. In this case $\Phi=\Phi_{\ell}$ is called \defn{simply-laced}.
\end{remark}

Consider the finite Coxeter group $W_0 := W_{\Phi}$ and let
$$W_0= \langle s_{\alpha_1}, \ldots , s_{\alpha_n} \rangle$$ 
for some roots $\alpha_i \in \Phi$. Let $W_a=W_{a,\Phi}$ be the associated affine Coxeter group. Furthermore we assume 
\begin{align} \label{equ:AffineRefGroup}
W_a = \langle s_{\alpha_1}, \ldots , s_{\alpha_n}, s_{\alpha_n,1} \rangle.
\end{align}
Note that we always find such a generating set. To see this, we start with a generating set $\{ s_{\alpha_1}, \ldots, s_{\alpha_n} \}$ of $W_0$, that is $s_{\alpha_1} \cdots s_{\alpha_n}$ is quasi-Coxeter. We can assume by Theorem \ref{th:maintheorem3} and Lemma \ref{le:EachRefBelow} that $\alpha_n=\widetilde{\alpha}$. Since the set $\{ s_{\alpha_1}, \ldots, s_{\alpha_n} \}$ generates $W_0$, the set $\{ s_{\alpha_1}, \ldots, s_{\alpha_n}, s_{\alpha_n, 1} \}$ generates $W_a$ by Theorem \ref{thm:StructureAffineWeylGroup}.

\begin{Lemma} \label{le:CrystValue}
\cite[Ch. VI, 1.3]{Bou02} Assume that $\Phi$ is not simply-laced. Let $\alpha, \beta \in \Phi$ be of different lengths with $(\alpha \mid \beta) \neq 0$ and let $\delta$ be the ratio of squared root lengths. Then 
\begin{equation*}
\frac{2(\alpha \mid \beta)}{(\beta \mid \beta)} =
\begin{cases} 
 \pm 1  &\mbox{if } \alpha \in \Phi_s, \beta \in \Phi_{\ell} \\
 \pm \delta & \mbox{if }  \alpha \in \Phi_{\ell}, \beta \in \Phi_s.
\end{cases}
\end{equation*}
\end{Lemma}

\medskip
\noindent Next we are going to show that the root $\alpha_n$ in equation (\ref{equ:AffineRefGroup}) cannot be short. Therefore note the following.
\begin{remark} \label{rem:OnShortLongRoots}
If we have in Equation (\ref{equ:AffineRefGroup}) that $\alpha_n \in \Phi_s, \alpha_i \in \Phi_{\ell}$ for some $i < n$ and $(\alpha_n \mid \alpha_i) \neq 0$, then by Lemma \ref{le:AffConj} and Lemma \ref{le:CrystValue} we have 
$$s_{\alpha_i}^{s_{\alpha_n, 1}} = s_{s_{\alpha_n}(\alpha_i), \pm \delta},$$
where $\delta$ is the ratio of squared root lengths. If $\alpha_n \in \Phi_{\ell}, \alpha_i \in \Phi_s$ we have 
$$s_{\alpha_i}^{s_{\alpha_n, 1}} = s_{s_{\alpha_n}(\alpha_i), \pm 1}.$$
\end{remark}


\medskip
\begin{Lemma} \label{le:PropSublattice}
Let $\Phi$ be an  irreducible crystallographic root system which consists of long and short roots and let $\delta$ be the ratio of squared root lengths. Then 
$$\spanz(\{ \delta \alpha \mid \alpha \in \Phi_s \} \cup \Phi_{\ell}) \subsetneq L(\Phi).$$
\end{Lemma}

\begin{proof}
We show that no short root is contained in $\spanz(\{ \delta \alpha \mid \alpha \in \Phi_s \} \cup \Phi_{\ell})$. Therefore let $\alpha \in \Phi_s$ be arbitrary. Assume that there exist $\alpha_1, \ldots, \alpha_n \in \Phi_s$ and $\beta_1 , \ldots , \beta_m \in \Phi_{\ell}$ such that 
$$\alpha = \sum_{i=1}^n \delta \alpha_i + \sum_{j=1}^m \beta_j.$$
Here we allow $\alpha_i = \alpha_j$ as well as $\beta_i = \beta_j$ for $i \neq j$ so that all coefficients in the above linear combination are equal to one. Thus we have 
\begin{align} \label{equ:a1a2}
(\alpha \mid \alpha) & = \sum_i \delta^2 (\alpha_i \mid \alpha_i) + \sum_{i < k} 2 \delta^2 (\alpha_i \mid \alpha_k) + \sum_i \sum_j 2\delta (\alpha_i \mid \beta_j) + \sum_j (\beta_j \mid \beta_j) +\sum_{j <l} 2 (\beta_j \mid \beta_l).
\end{align}
Let $\alpha' \in \Phi_s, \beta' \in \Phi_{\ell}$, thus $\delta=\frac{(\beta' \mid \beta')}{(\alpha' \mid \alpha')}$. Let $r \in \{ 1, \ldots, m \}$, then
\begin{align} \label{equ:a1a3}
\frac{2 \delta^2 (\alpha_i \mid \alpha_k)}{(\beta_r \mid \beta_r)} = \frac{2 \delta (\beta' \mid \beta')(\alpha_i \mid \alpha_k)}{(\alpha' \mid \alpha')(\beta_r \mid \beta_r)} = 
\delta \cdot \underbrace{\frac{2 (\alpha_i \mid \alpha_k)}{(\alpha' \mid \alpha')}}_{\in \ZZ} \in \ZZ.
\end{align}
Note that this is true for $\alpha_i = \alpha_k$ as well as for $\alpha_i \neq \alpha_k$. Therefore, if we divide equation (\ref{equ:a1a2}) by $(\beta_r \mid \beta_r)$, we obtain
\begin{align*}
\frac{1}{\delta} = \frac{(\alpha \mid \alpha)}{(\beta_r \mid \beta_r)} & = \sum_i \underbrace{\delta^2 \frac{(\alpha_i \mid \alpha_i)}{(\beta_r \mid \beta_r)}}_{=\delta } + \sum_{i < k} \underbrace{\frac{2 \delta^2 (\alpha_i \mid \alpha_k)}{(\beta_r \mid \beta_r)}}_{\in \ZZ ~\text{by (\ref{equ:a1a3})}} + \sum_i \sum_j \delta \cdot \underbrace{\frac{2 (\alpha_i \mid \beta_j)}{(\beta_r \mid \beta_r)}}_{\in \ZZ}\\
{} & + \sum_j \underbrace{\frac{(\beta_j \mid \beta_j)}{(\beta_r \mid \beta_r)}}_{=1} +\sum_{j <l} \underbrace{\frac{2 (\beta_j \mid \beta_l)}{(\beta_r \mid \beta_r)}}_{\in \ZZ}.
\end{align*}
Hence the right hand side of this equation is an integer, while the left hand side is not; a contradiction. 
\end{proof}


\begin{Lemma} \label{le:InvSublattice}
With the assumptions as in Lemma \ref{le:PropSublattice} we have that 
$$
L':=\spanz(\{ \delta \alpha^{\vee} \mid \alpha \in \Phi_{\ell} \} \cup \{ \alpha^{\vee} \mid \alpha \in \Phi_s \})
$$ is a proper sublattice of $L(\Phi^{\vee})$ and $w(\lambda) \in L' ~\text{for all } w \in W_0 ~\text{and for all } \lambda \in L'$.
\end{Lemma}

\begin{proof}
If we consider in the situation of Lemma \ref{le:PropSublattice} the dual root system $\Phi^{\vee}$ instead of $\Phi$, then we obtain that $L'$ is a proper sublattice of $L(\Phi^{\vee})$. It is sufficient to show the remaining assertion for the generators of $L'$. If $\alpha \in \Phi_{\ell}$, then $w(\delta \alpha^{\vee})= \delta w(\alpha^{\vee}) = \delta w(\alpha)^{\vee}$. Furthermore we have $w(\alpha) \in \Phi_{\ell}$ and therefore $w(\delta \alpha^{\vee}) \in L'$. Similarly, we obtain $w(\alpha^{\vee})= w(\alpha)^{\vee} \in \Phi_s$ if $\alpha \in \Phi_s$. 
\end{proof}



\begin{Proposition} \label{prop:ShortRootWeak}
Let $\Phi$ be an irreducible crystallographic root system of rank $n$ which is not simply-laced. If there exist roots $\alpha_1 , \ldots , \alpha_n \in \Phi$ such that 
$$W_{a, \Phi} = \langle s_{\alpha_1}, \ldots , s_{\alpha_n}, s_{\alpha_n,1} \rangle,$$
then $\alpha_n \in \Phi_{\ell}$. 
\end{Proposition}

\begin{proof}
As before put $W_0 := W_{\Phi}$. Assume the root $\alpha_n$ to be short. By Lemma \ref{le:CalcAffine} we have $s_{\alpha_n}s_{\alpha_n ,1}= \TR(-\alpha_n^{\vee})$ and therefore 
$$
W_{a, \Phi} = \langle s_{\alpha_1}, \ldots , s_{\alpha_n}, \TR(\alpha_n^{\vee}) \rangle.
$$
Let $L'$ be the proper sublattice of $L(\Phi^{\vee})$ as defined in Lemma \ref{le:InvSublattice}. We obtain $\TR(\alpha_n^{\vee}) \in W_0 \ltimes L'$, thus $W_0 \ltimes L'= W_{a, \Phi}$. But by Lemma \ref{le:PropSublattice} we have that $L'$ is a proper sublattice of $L(\Phi^{\vee})$ and therefore $W_0 \ltimes L'$ is a proper subgroup of $W_0 \ltimes L(\Phi^{\vee})= W_{a, \Phi}$, a contradiction.


\end{proof}


\begin{Proposition} \label{prop:FiniteInAffine}
Let $\Phi$ be an irreducible crystallographic root system of rank $n$ and $W_0=W_{\Phi}$. If there exist roots $\{ \beta_1, \ldots , \beta_n \} \subseteq \Phi$ such that $W_0= \langle s_{\beta_1}, \ldots , s_{\beta_n} \rangle$, then for any integers $k_1, \ldots , k_n \in \ZZ$ we have
$$ W_0 \cong \langle s_{\beta_1, k_1}, \ldots , s_{\beta_n, k_n} \rangle. $$
\end{Proposition}

\begin{proof}
Since the roots $\beta_1, \ldots , \beta_n$ have to be linearly independent, the corresponding hyperplanes intersect in one point. Therefore the group $W':=\langle s_{\beta_1, k_1}, \ldots , s_{\beta_n, k_n} \rangle$ is finite by \cite[Ch.V, 3, Prop. 4]{Bou02}. We consider again the map $p: W_{a, \Phi} \rightarrow W_0, s_{\alpha, k} \mapsto s_{\alpha}$. Let $w = s_{\alpha_{i_1},k_{i_1}} \cdots s_{\alpha_{i_m}, k_{i_m}} \in \ker(p_{\vert W'})$, thus $p(w)= s_{\alpha_{i_1}} \cdots s_{\alpha_{i_m}} = e$. Considering $w$ in its normal form $w= w_0 \TR(\lambda)$ with $w_0 \in W_0$ and $\lambda \in L(\Phi^{\vee})$, we have $w_0 = p(w) = e$. Hence $w$ has to be a translation. Therefore $\lambda = 0$ since $W'$ is finite. Thus $w = e$. By the first isomorphism theorem we obtain $W_0 \cong W'$. 
\end{proof}


\begin{Proposition} \label{prop:ShortRoot}
Let $\Phi$ be an irreducible crystallographic root system of rank $n$ which is not simply-laced. If there exist roots $\alpha_1 , \ldots , \alpha_n \in \Phi$ and integers $k_1, \ldots, k_{n-1} \in \ZZ$ such that 
$$W_{a, \Phi} = \langle s_{\alpha_1,k_1}, \ldots , s_{\alpha_{n-1}, k_{n-1}} , s_{\alpha_n}, s_{\alpha_n,1} \rangle,$$
then $\alpha_n \in \Phi_{\ell}$. 
\end{Proposition}

\begin{proof}
First of all note that $W_{\Phi}=p(W_{a, \Phi})=\langle s_{\alpha_1}, \ldots, s_{\alpha_n} \rangle$. Therefore the roots $\alpha_1 , \ldots , \alpha_n$ are linearly independent. Hence for fixed $\ell_1, \ell_2 , \ldots, \ell_n \in \ZZ$, the system of equations
\begin{align*}
(\lambda \mid \alpha_1) & = \ell_1\\
(\lambda \mid \alpha_2) & = \ell_2\\
{}                       & \vdots\\
(\lambda \mid \alpha_n) & = \ell_n,
\end{align*}
has a unique solution $\lambda \in V$. By Theorem \ref{PropVoigt} the roots $\alpha_1, \ldots, \alpha_n$ are a basis of $L(\Phi)$. Since $(\lambda \mid \alpha_i) = \ell_i \in \ZZ$ for all $1 \leq i \leq n$, we have by definition $\lambda \in P(\Phi^{\vee})$. Thus we obtain by part $(b)$ of Lemma \ref{le:CalcAffine2}:
\begin{align*}
\langle s_{\alpha_1,k_1}, \ldots , s_{\alpha_{n-1}, k_{n-1}} , s_{\alpha_n}, s_{\alpha_n,1} \rangle
 = & \TR(\lambda) \langle s_{\alpha_1,k_1}, \ldots , s_{\alpha_{n-1}, k_{n-1}} , s_{\alpha_n}, s_{\alpha_n,1} \rangle \TR(- \lambda)\\
{} = & \langle s_{\alpha_1,k_1+\ell_1}, \ldots , s_{\alpha_{n-1}, k_{n-1}+\ell_{n-1}} , s_{\alpha_n, \ell_n}, s_{\alpha_n,1+ \ell_n} \rangle.
\end{align*}
If we choose $\ell_i=-k_i$ for $1 \leq i \leq n-1$ and $\ell_n=0$, the assertion follows by Proposition \ref{prop:ShortRootWeak}. 
\end{proof}

\begin{Lemma} \label{le:Stefan}
Let $\Phi$ be an irreducible crystallographic root system of rank $n$. If 
$$W_{a, \Phi} = \langle s_{\alpha_1, \ell_1}, \ldots, s_{\alpha_{n-1}, \ell_{n-1}}, s_{\alpha_{n}, \ell_{n}}, s_{\alpha_{n}, \ell_{n+1}} \rangle$$
for roots $\alpha_1, \ldots, \alpha_n \in \Phi$ and integers $\ell_1, \ldots, \ell_n, \ell_{n+1} \in \ZZ$, then $\vert \ell_{n+1} - \ell_{n} \vert =1$.
\end{Lemma}

\begin{proof}
Consider the subgroup $U:=  \langle s_{\alpha_1, \ell_1}, \ldots, s_{\alpha_{n-1}, \ell_{n-1}}, s_{\alpha_{n}, \ell_{n}} \rangle $ of $W_{a, \Phi}$. By Lemma \ref{le:CalcAffine} we have $s_{\alpha_n, \ell_n}s_{\alpha_n, \ell_{n+1}}=\TR((\ell_n - \ell_{n+1})\alpha_n^{\vee})$. Hence 
\begin{align} \label{equ:Stefan1}
W_{a, \Phi}= \langle U, \TR((\ell_n - \ell_{n+1})\alpha_n^{\vee}) \rangle =  \langle s_{\alpha_1, \ell_1}, \ldots, s_{\alpha_{n-1}, \ell_{n-1}}, s_{\alpha_{n}, \ell_{n}}, \TR((\ell_n - \ell_{n+1})\alpha_n^{\vee}) \rangle.
\end{align}
Using the same argument as in the proof of Proposition \ref{prop:ShortRoot}, we take $\lambda \in P(\Phi^{\vee})$ to be the solution of the system of equations $(\lambda \mid \alpha_i)  = -\ell_i$ ($1 \leq i \leq n$).
Thus by part $(b)$ of Lemma \ref{le:CalcAffine2} we obtain for $1 \leq i \leq n$:
$$
\TR(\lambda) s_{\alpha_i, \ell_i} \TR(-\lambda)= s_{\alpha_i, \ell_i+(\lambda \mid \alpha_i)} = s_{\alpha_i}.
$$
Since two translations commute, we obtain again by part $(b)$ of Lemma \ref{le:CalcAffine2} and by (\ref{equ:Stefan1}) that 
\begin{align} \label{equ:Stefan2}
W_{a, \Phi}  = \TR(\lambda) \langle s_{\alpha_1, \ell_1}, \ldots , s_{\alpha_{n}, \ell_{n}}, \TR((\ell_n - \ell_{n+1})\alpha_n^{\vee}) \rangle \TR(- \lambda)
= \langle s_{\alpha_1}, \ldots , s_{\alpha_{n}}, \TR((\ell_n - \ell_{n+1})\alpha_n^{\vee}) \rangle.
\end{align}
Put $\ell:=\ell_n- \ell_{n+1}$. By Lemma \ref{le:CalcAffine} we have $s_{\alpha_n, \ell} s_{\alpha_n} = \TR(\ell \alpha_n^{\vee})$ and therefore (\ref{equ:Stefan2}) yields
\begin{align} \label{equ:Stefan3}
W_{a, \Phi} = \langle s_{\alpha_1}, \ldots , s_{\alpha_{n}}, \TR(\ell \alpha_n^{\vee}) \rangle =\langle s_{\alpha_1}, \ldots , s_{\alpha_{n}}, s_{\alpha_n, \ell} \rangle.
\end{align}
Put $R:= \{ s_{\alpha_1}, \ldots , s_{\alpha_{n}}, s_{\alpha_n, \ell} \}$ and $T':= \cup_{w \in W} w R w^{-1}$. Since we can write each element of $W_{a, \Phi}$ as a product of elements in $R$, Lemma \ref{le:AffConj} yields that if $s_{\alpha, k } \in T'$ ($\alpha \in \Phi$, $k \in \ZZ$), then $k$ has to be a multiple of $\ell$. Using (\ref{equ:Stefan3}) we obtain by \cite[Corollary 3.11]{Dye90} that 
$$T'= T_a= \{s_{\alpha, k} \mid \alpha \in \Phi, ~k \in \ZZ \}.$$
Hence, if we assume that $\vert \ell \vert >1$, we arrive at a contradiction, because in this case we have $T' \subsetneq T_a$. Therefore $\vert \ell \vert \leq 1$ and \cite[Proposition 2.1]{BGRW17} yields $1 = \vert \ell \vert = \vert \ell_n - \ell_{n+1} \vert$.

\end{proof}

Theorem \ref{cor:ParWallLongRoot} is now a direct consequence of Proposition \ref{prop:ShortRoot} and Lemma \ref{le:Stefan}.

\section{Proof of the main result} \label{sec:ProofMainResult}

The aim of this section is to prove the Main Theorem \ref{th:themaintheorem1}. Therefore we fix for the rest of this section an affine irreducible dual Coxeter system $(W_a,T_a)$ of rank $n+1$ with $W_a= W_{a, \Phi}$ for some irreducible crystallographic root system $\Phi$ of rank $n$. Let $W_0:= W_{\Phi}$ be the underlying finite Coxeter group and put $T_0 := \{s_{\alpha} \mid \alpha \in \Phi \}$.

Let $w$ be a quasi-Coxeter element for $(W_a,T_a)$. We already noted in (\ref{equ:ThisIsQuasiCox}) in Section \ref{sec:QuasiCoxElInAffine} that the element $w':=p(w)$ is a parabolic quasi-Coxeter element for $(W_0,T_0)$ of absolute length $n-1$. One of our main tools to prove Theorem \ref{th:themaintheorem1}, that is, to prove that the Hurwitz action is transitive on the set $\Red_{T_a}(w)$, will be the investigation of the Hurwitz action on the set 
$$ \Fac_{T_0, n+1}(w')= \{ (t_1, \ldots , t_{n+1}) \in T_0^{n+1} \mid ~t_1 \cdots t_{n+1}=w',~\langle t_1, \ldots, t_{n+1} \rangle=W_0 \}.$$
Note that this set does not contain reduced reflection decompositions for $w'$. Consider the map 
$$ \pi: \Red_{T_a}(w) \rightarrow \Fac_{T_0,n+1}(w'), (r_1, \ldots, r_{n+1}) \mapsto (p(r_1), \ldots, p(r_{n+1})).$$
The three main steps to prove Theorem \ref{th:themaintheorem1} will be to show the following assertions:
\begin{itemize}
\item The map $\pi$ is well-defined.
\item The Hurwitz action is transitive on $\Fac_{T_0, n+1}(w')$.
\item For each $\underline{r}= (r_1, \ldots, r_{n-1}, r_n, r_n) \in \Fac_{T_0, n+1}(w')$ there exists a subgroup of the isotropy subgroup $\text{Stab}_{\mathcal{B}_{n+1}}(\underline{r})$ which acts transitively on the fibre $\pi^{-1}(\underline{r})$.  
\end{itemize}
 
\medskip
\begin{Proposition} \label{prop:WellDefined}
The map $\pi$ is well-defined, that is, if $(s_{\beta_1, k_1}, \ldots, s_{\beta_{n+1}, k_{n+1}}) \in \Red_{T_a}(w)$, then $\langle s_{\beta_1}, \ldots, s_{\beta_{n+1}} \rangle =W_0$. 
\end{Proposition}

\begin{proof}
By equation (\ref{equ:AffineQuasi}) in Section \ref{sec:QuasiCoxElInAffine} there exists a reflection decomposition of $w$ of the form  
$$(s_{\gamma_1, \ell_1}, \ldots, s_{\gamma_{n-1}, \ell_{n-1}}, s_{\gamma_{n}, \ell_{n}}, s_{\gamma_{n+1}, \ell_{n+1}}),$$
where $\gamma_{n+1} = \gamma_{n}$ and $W_a= \langle s_{\gamma_1, \ell_1}, \ldots, s_{\gamma_{n}, \ell_{n}}, s_{\gamma_{n+1}, \ell_{n+1}} \rangle$. Let $(s_{\beta_1, k_1}, \ldots, s_{\beta_{n+1}, k_{n+1}})$ be an arbitrary element of $\Red_{T_a}(w)$ and put $w':=p(w)$. By Proposition \ref{prop:SameRootsOrbit} we can assume that $\beta_{n+1}=\beta_n$. We have
\begin{align*}
w_0:= w' s_{\gamma_{n}} & = \underbrace{s_{\gamma_1} \cdots s_{\gamma_{n-1}}}_{=w'} s_{\gamma_{n}}
= \underbrace{s_{\beta_1} \cdots s_{\beta_{n-1}}}_{=w'} s_{\gamma_{n}}.
\end{align*}
Since $\gamma_{n+1} = \gamma_{n}$, we have $W_0=p(W_a)=\langle s_{\gamma_1}, \ldots, s_{\gamma_n} \rangle$, that is, $w_0$ is quasi-Coxeter. Hence we have 
$$W_0= \langle s_{\gamma_1}, \ldots, s_{\gamma_{n}} \rangle = \langle s_{\beta_1}, \ldots, s_{\beta_{n-1}}, s_{\gamma_{n}} \rangle.$$ 
Since $w'$ is a parabolic quasi-Coxeter element, Theorem \ref{th:maintheorem3} yields
$$ (s_{\beta_1}, \ldots , s_{\beta_{n-1}}) \sim (s_{\gamma_1}, \ldots , s_{\gamma_{n-1}}).$$
Therefore we assume without loss of generality that $\gamma_i = \beta_i$ for $1 \leq i \leq n-1$. By the equivalence of $(b)$ and $(c)$ in Theorem \ref{PropVoigt} it remains to show:
\begin{itemize}
\item[(i)] $\gamma_{n} \in L(\beta_1, \ldots, \beta_{n}) = L(\gamma_1, \ldots, \gamma_{n-1}, \beta_{n})$;
\item[(ii)] $\gamma_{n}^{\vee} \in L(\beta_1^{\vee}, \ldots, \beta_{n}^{\vee})= L(\gamma_1^{\vee}, \ldots, \gamma_{n-1}^{\vee}, \beta_{n}^{\vee})$.
\end{itemize} 
Writing $w$ in its normal form as described in Lemma \ref{le:AffineFac}, we obtain 
\begin{align*}
w & = w' \TR \left(\sum_{i=1}^{n+1} -\ell_i s_{\gamma_{n+1}} \cdots s_{\gamma_{i+1}}(\gamma_i)^{\vee} \right)\\
{} & =  w' \TR \left( \sum_{i=1}^{n+1} -k_i s_{\beta_{n+1}} \cdots s_{\beta_{i+1}}(\beta_i)^{\vee} \right).
\end{align*}
Hence both translation parts must be equal. By Theorem \ref{cor:ParWallLongRoot} we have $\ell_n - \ell_{n+1} = \pm 1$ and $\gamma_n$ is a long root. Using the facts that $\gamma_i = \beta_i$ for $1 \leq i \leq n-1$ and $\gamma_{n+1}=\gamma_n$, $\beta_{n+1}=\beta_n$, we obtain
$$  \sum_{i=1}^{n-1} -\ell_i s_{\gamma_{n-1}} \cdots s_{\gamma_{i+1}}(\gamma_i)^{\vee} + \underbrace{(\ell_{n}-\ell_{n+1})}_{= \pm 1} \gamma_{n}^{\vee} = 
\sum_{i=1}^{n-1} -k_i s_{\gamma_{n-1}} \cdots s_{\gamma_{i+1}}(\gamma_i)^{\vee} + (k_{n}-k_{n+1})\beta_{n}^{\vee},$$
thus 
$$\pm \gamma_{n}^{\vee} = \sum_{i=1}^{n-1} (\ell_i-k_i) s_{\gamma_{n-1}} \cdots s_{\gamma_{i+1}}(\gamma_i)^{\vee} + (k_{n}-k_{n+1})\beta_{n}^{\vee}.$$
Statement (ii) follows by Lemma \ref{le:LinIndCoroots}.
It remains to show (i). The root $\gamma_{n+1}= \gamma_n$ is a long root. By (ii) we have 
$$\gamma_{n}^{\vee} = \sum_{i=1}^{n} \lambda_i \beta_i^{\vee}$$
for some integer coefficients $\lambda_i \in \ZZ$. Furthermore we have
\begin{equation*}
\frac{(\gamma_{n} \mid \gamma_{n})}{(\beta_i \mid \beta_i)} \in
\begin{cases} 
  \{ 1 \}  &\mbox{if } \beta_i ~\text{is long} \\
  \{ 2,3 \} & \mbox{if }  \beta_i ~\text{is short},
\end{cases}
\end{equation*}
see also Remark \ref{rem:RootRatio}. Therefore we obtain
$$\gamma_{n} = \sum_{i=1}^{n} \frac{(\gamma_{n} \mid \gamma_{n})}{(\beta_i \mid \beta_i)}\lambda_i \beta_i \in L(\beta_1, \ldots, \beta_{n}).$$
\end{proof}
 
\begin{Proposition} \label{prop:equi}
The map $\pi$ is equivariant with respect to the Hurwitz action.
\end{Proposition}

\begin{proof}
By Lemma \ref{le:AffConj} we have 
$$
\sigma_i (\ldots, s_{\alpha_i, k_i}, s_{\alpha_{i+1}, k_{i+1}}, \ldots) = 
\left(\ldots, s_{s_{\alpha_{i}}(\alpha_{i+1}),k_{i+1}-\frac{2(\alpha_i \mid \alpha_{i+1})}{(\alpha_{i} \mid \alpha_{i})}k_{i} }, s_{\alpha_i}, \ldots \right).
$$
Hence
$$
\pi(\sigma_i (\ldots, s_{\alpha_i, k_i}, s_{\alpha_{i+1}, k_{i+1}}, \ldots))= (\ldots, s_{s_{\alpha_{i}}(\alpha_{i+1})}, s_{\alpha_i}, \ldots) = \sigma_i (\ldots, s_{\alpha_i}, s_{\alpha_{i+1}}, \ldots).
$$
\end{proof}

\begin{theorem} \label{thm:HurwitzOnFac}
The Hurwitz action is transitive on $\Fac_{T_0,n+1}(w')$. 
\end{theorem}

\begin{proof}
Let $(t_1, \ldots, t_{n+1}), (r_1, \ldots, r_{n+1}) \in \Fac_{T_0,n+1}(w')$ be arbitrary. By Theorem \ref{thm:LR2} we can assume (up to Hurwitz equivalence) that $t_{n+1}=t_n$ and $r_{n+1}=r_n$. It is $w'=t_1 \cdots t_{n-1}= r_1 \cdots r_{n-1}$ a parabolic quasi-Coxeter element with corresponding parabolic subgroup $W'= \langle t_1, \ldots, t_{n-1} \rangle \leq W_0$. In particular we have 
$$(t_1, \ldots, t_{n-1}) \sim (r_1, \ldots, r_{n-1})$$
by Theorem \ref{th:maintheorem3}, hence  
$$(r_1, \ldots, r_{n-1}, r_{n}, r_{n}) \sim (t_1, \ldots, t_{n-1}, r_{n}, r_{n}).$$
By Theorem \ref{prop:CharParSub} the reflections $r_n$ and $t_n$ are conjugated under $W'$. The assertion follows by Lemma \ref{le:Reduction}.
\end{proof}

\medskip
\noindent Direct calculations yield the following statement.
\begin{Lemma} \label{le:HurwitzDihedral}
Let $\alpha \in \Phi$. Then 
$$(s_{\alpha, 1} , s_{\alpha, 0}) \sim (s_{\alpha, k+1} , s_{\alpha, k}) ~\text{and } (s_{\alpha, 0} , s_{\alpha, 1}) \sim (s_{\alpha, k} , s_{\alpha, k+1}) ~\text{for all } k \in \ZZ.$$
\end{Lemma}

\begin{Lemma} \label{le:DualRootBase}
Let $\{ \alpha_1 , \ldots, \alpha_n\} \subseteq \Phi$ such that $W_0= \langle s_{\alpha_1}, \ldots, s_{\alpha_n} \rangle$. Then the set 
$$\{ s_{\alpha_{n-1}} \cdots s_{\alpha_{i+1}} (\alpha_i)^{\vee} \mid 1 \leq i \leq n-1 \} \cup \{ \alpha_n^{\vee} \}$$
is a basis of $L(\Phi^{\vee})$.
\end{Lemma}

\begin{proof}
We have $W_0= \langle s_{s_{\alpha_{n-1}} \cdots s_{\alpha_2}(\alpha_1)}, \ldots ,s_{s_{\alpha_{n-1}}(\alpha_{n-2})}, s_{\alpha_{n-1}}, s_{\alpha_n} \rangle$ and therefore the assertion is a direct consequence of Theorem \ref{PropVoigt}.
\end{proof}

\begin{Lemma} \label{le:Fibre}
Let $w$ be a quasi-Coxeter element for $(W_a,T_a)$. Then $\langle \sigma_{n} \rangle \subseteq \mathcal{B}_{n+1}$ acts transitively on the fibre $\pi^{-1}(p(\underline{r}))$ for each $\underline{r}=(s_{\gamma_1, \ell_1}, \ldots, s_{\gamma_{n-1}, \ell_{n-1}}, s_{\gamma_{n}, \ell_{n}}, s_{\gamma_{n}, \ell_{n+1}}) \in \Red_{T_a}(w)$.
\end{Lemma}

\begin{proof}
Fix $\underline{r}=(s_{\gamma_1, \ell_1}, \ldots, s_{\gamma_{n-1}, \ell_{n-1}}, s_{\gamma_{n}, \ell_{n}}, s_{\gamma_{n}, \ell_{n+1}}) \in \Red_{T_a}(w)$. Clearly an element in the fibre has to be of the form 
$$
(s_{\gamma_1, m_1}, \ldots, s_{\gamma_{n-1}, m_{n-1}}, s_{\gamma_{n}, m_{n}}, s_{\gamma_{n}, m_{n+1}}) \in \Red_{T_a}(w)
$$
for integers $m_1, \ldots, m_{n+1} \in \ZZ$. Considering the normal form corresponding to this reflection decomposition and the normal form corresponding to the reflection decomposition $\underline{r}$, we have equality of the translation parts. By the same calculations as in the proof of Proposition \ref{prop:WellDefined} we obtain
$$
\sum_{i=1}^{n-1} -l_i s_{\gamma_{n-1}} \cdots s_{\gamma_{i+1}}(\gamma_i)^{\vee} + (\ell_{n}-\ell_{n+1}) \gamma_{n}^{\vee} = \sum_{i=1}^{n-1} -m_i s_{\gamma_{n-1}} \cdots s_{\gamma_{i+1}}(\gamma_i)^{\vee} + (m_{n}-m_{n+1}) \gamma_{n}^{\vee}.
$$
By Proposition \ref{prop:WellDefined} we have $W_0= \langle s_{\gamma_1}, \ldots, s_{\gamma_n} \rangle$. The roots $\gamma_1^{\vee}, \ldots, \gamma_{n-1}^{\vee}, \gamma_{n}^{\vee}$ are a basis of $L(\Phi^{\vee})$ by Theorem \ref{PropVoigt}. Therefore by Lemma \ref{le:DualRootBase} the set 
$$\{ s_{\gamma_{n-1}} \cdots s_{\gamma_{2}}(\gamma_1)^{\vee}, \ldots, s_{\gamma_{n-1}}(\gamma_{n-2})^{\vee}, \gamma_{n-1}^{\vee}, \gamma_{n}^{\vee} \}$$ 
is another basis. This yields that $m_i=\ell_i$ for all $i \in \{1, \ldots, n-1\}$ and $m_{n}-m_{n+1} = \ell_{n}-\ell_{n+1}$. Furthermore Lemma \ref{le:Stefan} yields $m_{n}-m_{n+1} = \ell_{n}-\ell_{n+1}= \pm 1$. By these properties and by Lemma \ref{le:HurwitzDihedral} we conclude that $\langle \sigma_{n} \rangle \subseteq \mathcal{B}_{n+1}$ acts transitively on $\pi^{-1}(\underline{r})$. 
\end{proof}



\begin{remark}
In fact, the fibre that we considered in Lemma \ref{le:Fibre} can be described completely in terms of the translation part of $s_{\gamma_{n}, \ell_{n}}$ and $s_{\gamma_{n}, \ell_{n+1}}$ (resp. its coefficients $\ell_{n}$ and $\ell_{n+1}$). For $\ell_{n}-\ell_{n+1}=1$, we have 
$$\ldots \sim (-2,-1) \sim (-1, 0) \sim (0,1) \sim (1,2) \sim (2,3) \sim \ldots $$
For $\ell_{n}-\ell_{n+1}=-1$, we have 
$$\ldots \sim (-1,-2) \sim (0, -1) \sim (1,0) \sim (2,1) \sim (3,2) \sim \ldots $$
\end{remark}

\bigskip
\noindent We are finally in the position to prove our main result.
\begin{proof}[Proof of Theorem~\ref{th:themaintheorem1}]
The reduction to the irreducible case is immediate (see also \cite[Lemma 8.1]{BGRW17}). Therefore we proceed with our fixed affine dual Coxeter system $(W_a,T_a)$ of rank $n+1$. Let $w \in W_a$ be a quasi-Coxeter element and fix a reduced reflection decomposition 
$$(s_{\gamma_1, \ell_1}, \ldots, s_{\gamma_{n-1}, \ell_{n-1}}, s_{\gamma_{n}, \ell_{n}}, s_{\gamma_{n}, \ell_{n+1}}) \in \Red_{T_a}(w)$$
which is obtained as in equation (\ref{equ:AffineQuasi}) in Section \ref{sec:QuasiCoxElInAffine}. Let $(t_1, \ldots , t_{n+1}) \in \Red_{T_a}(w)$ be arbitrary. By Theorem \ref{thm:HurwitzOnFac} and Proposition \ref{prop:equi} there exists a braid $\sigma \in \mathcal{B}_{n+1}$ such that 
$$\pi(\sigma(t_1, \ldots , t_{n+1})) = \pi(s_{\gamma_1, \ell_1}, \ldots, s_{\gamma_{n-1}, \ell_{n-1}}, s_{\gamma_{n}, \ell_{n}}, s_{\gamma_{n}, \ell_{n+1}}).$$
Hence $\sigma(t_1, \ldots , t_n)$ and $(s_{\gamma_1, \ell_1}, \ldots, s_{\gamma_{n-1}, \ell_{n-1}}, s_{\gamma_{n}, \ell_{n}}, s_{\gamma_{n}, \ell_{n+1}})$ are in the same fibre and the assertion follows by Lemma \ref{le:Fibre}.

If $w \in W$ is a parabolic quasi-Coxeter element, but not a quasi-Coxeter element, then the assertion follows by Theorem \ref{th:maintheorem3} and Theorem \ref{th:maintheorem2} since all proper parabolic subgroups of $(W_a,T_a)$ are finite.
\end{proof}
 
\medskip
\begin{remark} \label{rem:end}
Let $(W,T)$ be a dual Coxeter system of rank $n$. By \cite[Proposition 2.1]{BGRW17} a quasi-Coxeter element has to be at least of absolute length $n$. If $(W,T)$ is finite, then by \cite[Lemma 3]{Carter} the absolute length is bounded by $n$. Hence it is canonical to demand a quasi-Coxeter element to be of absolute length $n$. If $(W,T)$ is not finite, then the absolute length is in general not bounded by above (see \cite{Dus12}), except for the case where $(W,T)$ is affine. In that case it is bounded by $2(n-1)$ (see \cite{McCP11}). Therefore it makes sense to ask whether one can extend the definition of quasi-Coxeter element to elements of absolute length greater than $n$. Namely it might make sense to define an element $w \in W$ to be quasi-Coxeter if there exists $(t_1, \ldots, t_m) \in \Red_T(w)$ with $m\geq n$ such that $W= \langle t_1, \ldots, t_m \rangle$ and the Hurwitz action is transitive on $\Red_T(w)$. For $(W,T)$ finite this definition is equivalent to Definition \ref{def:coxeter} (b). Using this more general definiton we did not find a quasi-Coxeter element such that the Hurwitz action on $\Red_T(w)$ is not transitive. The following example provides a quasi-Coxeter element for this more general definiton with $m>n$ and with transitive Hurwitz action. It was proposed by Thomas Gobet. 
\end{remark}

\begin{example} \label{ex:Thomas}
Let $(W,T)$ be affine of type $\widetilde{A}_2$. We choose a simple system $S \subseteq T$ such that $S= \{ s_{\alpha_1}, s_{\alpha_2}, s_{\widetilde{\alpha}, 1} \}$, where $\alpha_1, \alpha_2$ are simple roots for the corresponding root system of type $A_2$ and $\widetilde{\alpha}=\alpha_1+\alpha_2$ is the highest root. We consider the element $w:=(s_{\alpha_1}s_{\alpha_2}s_{\widetilde{\alpha},1})^2$. Since this is the power of a Coxeter element, we have $\ell_{S}(w)=6$ by \cite{Spe09}, where $\ell_S(w)$ denotes the length of $w$ with respect to the generating set $S$. Using the criterion given by Dyer in \cite[Theorem 1.1]{Dye01}, we see that $\ell_{T}(w)=4$. Note that $4$ is precisely the upper bound for the absolute length in $\widetilde{A}_2$. A reduced reflection decomposition is given by 
$$
w= s_{\widetilde{\alpha}} s_{\alpha_2, 1} s_{\alpha_2} s_{\widetilde{\alpha}, 1}
$$ 
and 
$$(s_{\widetilde{\alpha}}, s_{\alpha_2, 1}, s_{\alpha_2}, s_{\widetilde{\alpha}, 1}) \sim (s_{\widetilde{\alpha}, 1}, s_{\widetilde{\alpha}}, s_{\alpha_2, 1}, s_{\alpha_2}).$$
Therefore we see that $w= \TR(\alpha_1 + 2 \alpha_2)$. By Theorem \ref{thm:LR2} a reduced reflection decomposition of $w$ (up to Hurwitz equivalence) is given by $w=s_{\alpha, k_1} s_{\alpha, \ell_1} s_{\beta, k_2} s_{\beta, \ell_2}$ with $\alpha, \beta \in \{ \alpha_1, \alpha_2, \widetilde{\alpha}\}$. Note that $\alpha \neq \beta$ since otherwise $s_{\alpha, k_1} s_{\alpha, \ell_1} s_{\beta, k_2} \in T$ by Lemma \ref{le:AffineFac}, which would contradict $\ell_{T}(w)=4$. We have
\begin{align*}
(s_{\alpha_1}, s_{\alpha_1}, s_{\alpha_2}, s_{\alpha_2}) \stackrel{\sigma_2 \sigma_1 \sigma_3 \sigma_2}{\sim} (s_{\alpha_2}, s_{\alpha_2}, s_{\alpha_1}, s_{\alpha_1})
\end{align*}
and 
\begin{align*}
(s_{\alpha_1}, s_{\alpha_1}, s_{\alpha_2}, s_{\alpha_2}) \stackrel{}{\sim} (s_{\alpha_2}, s_{\alpha_1}, s_{\alpha_1}, s_{\alpha_2 }) \sim (s_{\alpha_2}, s_{\alpha_2}, s_{\widetilde{\alpha}}, s_{\widetilde{\alpha}} ) \stackrel{\sigma_2 \sigma_1 \sigma_3 \sigma_2}{\sim} (s_{\widetilde{\alpha}}, s_{\widetilde{\alpha}}, s_{\alpha_2}, s_{\alpha_2})
\end{align*}
and
\begin{align*}
(s_{\alpha_1}, s_{\alpha_1}, s_{\alpha_2}, s_{\alpha_2}) \stackrel{}{\sim} (s_{\alpha_1}, s_{\alpha_2}, s_{\alpha_2}, s_{\alpha_1 }) \sim (s_{\widetilde{\alpha}}, s_{\widetilde{\alpha}}, s_{\alpha_1}, s_{\alpha_1} ) \stackrel{\sigma_2 \sigma_1 \sigma_3 \sigma_2}{\sim} (s_{\alpha_1}, s_{\alpha_1}, s_{\widetilde{\alpha}}, s_{\widetilde{\alpha}}).
\end{align*}
By comparison of coefficients and Lemma \ref{le:HurwitzDihedral} we obtain Hurwitz transitivity on $\Red_{T}(w)$ by using similar arguments as before in this section. 
\end{example}

  \end{document}